\documentclass[11pt,reqno]{amsart}
\usepackage{amsfonts,amssymb,amsmath}
\usepackage{enumitem}
\usepackage{xcolor}
\usepackage{multirow}
\usepackage{float}

\definecolor{fgreen}{RGB}{44,144, 14}

\renewenvironment{proof}{{\bfseries Proof.}}{\qed}
\numberwithin{equation}{section} 
\newtheorem{theorem}{Theorem}[section] 
\newtheorem{proposition}[theorem]{Proposition} 
 
\newtheorem{lemma}[theorem]{Lemma} 

\theoremstyle{definition}
\newtheorem{definition}[theorem]{Definition} 
\newtheorem{remark}[theorem]{Remark} 
\newtheorem{example}[theorem]{Example}

\usepackage[colorlinks=true,citecolor=cyan, 
 urlcolor=blue,linkcolor=blue]{hyperref}

\def\R{\mathbb R}

\def\C{\mathbb C}

\def\F{\mathbb F}

\def\H{\mathbb H}

\def\ib{\mathbf {i}}
\def\jb{\mathbf {j}}
\def\kb{\mathbf {k}}

\def\D{\mathbb D}

\def\N{\mathbb N}

\def\R{\mathbb R}

\def\g{\mathcal G}

\newcommand{\GL}{\mathrm{GL}}

\def\R{\mathbb {R}}
\def\C{\mathbb {C}}
\def\N{\mathbb {N}}
\def\H{\mathbb {H}}

\def\F{\mathbb {F}}

\def\ib{\mathbf {i}}

\def\jb{\mathbf {j}}

\def\g{\mathfrak {g}}

\def\l{\mathfrak {l}}

\def\lto{\longrightarrow}
\def\GL{\rm GL}

\newcommand{\secref}[1]{Section~\ref{#1}}
\newcommand{\thmref}[1]{Theorem~\ref{#1}}
\newcommand{\lemref}[1]{Lemma~\ref{#1}}
\newcommand{\remref}[1]{Remark~\ref{#1}}
\newcommand{\propref}[1]{Proposition~\ref{#1}}

\begin{document}

\title[Reversibility  in $\mathrm{Aff}(n,\D) $  ]{Reversibility of Affine Transformations} 
 \author[K. Gongopadhyay,  T. Lohan   and  C. Maity]{Krishnendu Gongopadhyay, Tejbir Lohan   and  Chandan Maity}

\address{Indian Institute of Science Education and Research (IISER) Mohali,
 Knowledge City,  Sector 81, S.A.S. Nagar 140306, Punjab, India}
\email{krishnendug@gmail.com, krishnendu@iisermohali.ac.in}

 \address{Indian Institute of Science Education and Research (IISER) Mohali,
 Knowledge City,  Sector 81, S.A.S. Nagar 140306, Punjab, India}
\email{tejbirlohan70@gmail.com,  ph18028@iisermohali.ac.in}

\address{Indian Institute of Science Education and Research (IISER) Mohali,
 Knowledge City,  Sector 81, S.A.S. Nagar 140306, Punjab, India}
 \email{maity.chandan1@gmail.com, cmaity@iisermohali.ac.in }

\makeatletter
\@namedef{subjclassname@2020}{\textup{2020} Mathematics Subject Classification}
\makeatother



\begin{abstract}
 An element $g$ in a group $G$ is called  \emph{reversible} if $g$ is conjugate to $g^{-1}$ in $ G $. An element $g$ in $G$ is \emph{strongly reversible}   if $ g $  is conjugate to $g^{-1}$ by an involution in $G$. The group  of affine transformations of $\D^n$  may be identified with the semi-direct product $\mathrm{GL}(n, \D) \ltimes \D^n $, where $\D:=\R, \C$ or $ \H $.
  This paper classifies reversible and strongly reversible elements  in the affine group $\mathrm{GL}(n, \D) \ltimes \D^n $.
 
\end{abstract}

\maketitle 

\section{Introduction} \label{sec-intro-1}
Let $G$  be a group. An element $g \in G$ is called  {\it reversible}  or {\it real}  if $g$ is conjugate to $g^{-1}$ in $ G $. An element $g \in G$ is {\it strongly reversible}  or {\it strongly real}  if $ g $  is conjugate to $g^{-1}$ in $G$ by an involution  (i.e., by an element of order at most $2$) in $ G $. Equivalently, an element is strongly reversible if it is a product of two involutions from $G$; see \remref{rem-equvi-str}.
The idea of `reversible elements' originated in mathematical and physical systems from different directions, \cite{Ar, De, Se, lamb,  OS}.  From the algebraic point of view, the terms real and strongly real are used instead of reversible and strongly reversible. 
Investigation of reversible and strongly reversible elements in a group is an active area of current research; see \cite{OS} for an elaborate exposition of this theme from the geometric point of view.  
A complete classification of reversible and strongly reversible elements is not available in the literature  except for the case of a  few families of infinite groups, which include the compact Lie groups, real rank one classical groups, and isometry groups of hermitian spaces; see  \cite{OS, BG, GL}. In this article, by \textit{reversibility} in a group $G$,  we mean a classification of reversible and strongly reversible elements in $G$.

Let $\D :=\R,\C$ or $ \H$.  The space $ \D^n $  equipped with a (right) $\D$-Hermitian form gives a model for Hermitian geometry. When $ \D=\R $, this is the well-known classical Euclidean geometry. The reversibility problem in the isometry group  $ {\rm O}(n)\ltimes \R^n $ of the $n$-dimensional Euclidean space was classified by Short in \cite{Sh}. This has been extended in \cite{GL} for  the isometry group  $ {\rm U}(n,\F)\ltimes \F^n $ of the $\F$-Hermitian space, where $\F:= \C $ or $\H$.

Considering $\D^n $ as an affine space, the group  of automorphisms of $\D^n$, denoted by $  \mathrm{Aff}(n,\D) $,  is given by  $\mathrm{GL}(n,\D)\ltimes \D^n $.  
 The affine space is important to understand the affine structure on geometric manifolds; see the tome \cite{Go} for details. Understanding reversible and strongly reversible elements in the affine group $  \mathrm{Aff}(n,\D) $ is a natural problem of interest. In this paper, we have investigated this problem.
Our main result is the following: 

\begin{theorem}\label{thm-affine-main-theorem}
Let $g  = (A,v) \in  \mathrm{Aff}(n,\D) $ be an arbitrary element,  where  $ \D = \R, \C$ or $\H$.  Then $g$ is reversible (resp.  strongly reversible) in $\mathrm{Aff}(n,\D) $ if and only if $A$ is reversible (resp.  strongly reversible) in $\mathrm{GL}(n,\D)$. 
Further,  for $\D = \R $ or $\C$, the following statements are equivalent.
	\begin{enumerate}
		\item $g$ is  reversible in $\mathrm{Aff}(n,\D) $.
			\item $g$ is strongly reversible in $\mathrm{Aff}(n,\D) $.
	\end{enumerate}
\end{theorem}

This theorem answers a  problem raised in {\cite[p. 78--79]{OS}}. Note that the classification of the reversible and strongly reversible elements in $  \mathrm{Aff}(n,\D) $ is intimately related to the  corresponding classification  in  $\mathrm{GL}(n,\D)$. 
Such  classification  in $\mathrm{Aff}(n,\D)$ can be obtained by combining  \thmref{thm-affine-main-theorem} with the reversibility in $\mathrm{GL}(n,\D)$. The reversibility in ${\rm GL}(n, \D)$ is well-known for $\D=\R$ or $\C$, cf. \cite{Wo,  OS}, and this has been extended over the  quaternions recently, cf. \cite{GLM2}.

To prove the above theorem, first, we investigate conjugacy in ${\rm Aff}(n,\D)$ in \lemref{lem-affine-conj-form}. Then using \lemref{lem-affine-conj-form}, reversibility in ${\rm Aff}(n,\D)$ boils down to the case when the linear part of the affine transformation is unipotent. We consider the Lie algebra $\mathfrak{aff}(n, \D)$ of the affine group ${\rm Aff}(n,\D)$ and consider the adjoint action; see Equation \eqref{Ad-action-affine}.  Then we apply the notion of ``adjoint reality'' introduced in \cite{GM}, also see \secref{unisub},  to classify the strongly reversible elements in ${\rm Aff}(n,\D)$ whose linear parts are unipotent; see \propref{prop-affine-str-rev-unipotent}.

The reversibility problem is closely related to the problem of finding the involution length of a group. The {\it involution length} of a group $ G $ is the least integer $m$ so that any element of $ G $ can be expressed as a product of $ m $ involutions in $ G $; see {\cite[p. 76]{OS}}. Now we state our second result. We refer to Definition \ref{def-det} for the notion of quaternionic determinant.
\begin{theorem}\label{thm-affine-involution-lenght}
Let  $ g = (A,v) \in {\rm Aff}(n,\D) $ such that  $ \mathrm{det}(A) \in \{-1,1\}$. Then $g$ can be written as a product of at most four involutions for $ \D=\R, \C $ or $\H$.
\end{theorem}

\subsection{Structure of the paper}
The structure of the paper is as follows.  In  \secref{sec-prel-2}, we fix some notation and recall some necessary  background.
In  \secref{sec-affine-6}, we consider the affine group and prove the main result of this article, Theorem \ref{thm-affine-main-theorem}. Finally, in  \secref{sec-involution-length}, we investigate the product of involutions in the affine group $ {\rm Aff}(n,\D)$ and prove \thmref{thm-affine-involution-lenght}.

\section{Preliminaries} \label{sec-prel-2}
Let $\H:=\R\, +\,\R\ib +\,\R\jb+\,\R\kb$  be the division algebra of Hamilton’s quaternions. We will use the notation $\D$ to denote either $\R, \C$ or  $ \H$ unless otherwise specified. We consider $ \D^n$ as a right $\D$-module. We begin by recalling some basic notions of quaternion linear algebra. We refer the reader to \cite[Chapter 3, Chapter 5]{rodman} for a detailed exposition of the theory of linear transformations over the quaternions.

\begin{definition}[cf.~{\cite[p.  90]{rodman}}]\label{def-eigen-M(n,H)}
Let $  \mathrm{M}(n,\H)$ be the algebra of $n \times n$ matrices over $\H$. A non-zero vector $v \in \H^n $ is said to be a (right) eigenvector  of $A \in  \mathrm{M}(n,\H)$ corresponding to a (right) eigenvalue  $\lambda \in \H $ if the equality $ A v = v\lambda $ holds.
\end{definition}

Note that eigenvalues of $A\in  \mathrm{M}(n,\H)$ occur in similarity classes 
and each similarity class of eigenvalues contains a unique complex representative with  non-negative imaginary part. Here,  instead of similarity classes of eigenvalues, we will consider the unique complex representative with non-negative imaginary part.

\begin{definition}[cf.~{\cite[p.  94]{rodman}}] \label{def-jordan}
A {\it Jordan block} $\mathrm{J}(\lambda,m)$ is an $m \times m$ matrix with $ \lambda \in \D$ on the diagonal entries,  $1$ on all of the super-diagonal entries and zero elsewhere.  We will refer to a block diagonal matrix where each block is a Jordan block as  \textit{Jordan form}. 
\end{definition}

Jordan canonical forms in $\mathrm{GL}(n,\D)$ are well studied in the literature; see {\cite[Chapter 5,  Chapter 15]{rodman}}. Recall that	an element $U \in \mathrm{GL}(n,\D)$ is called unipotent if each eigenvalue of $U$ equals to $ 1$.  In our convention, we shall include identity as the only unipotent element, which is also semisimple. Next result provides Jordan form for a given unipotent element in $ \mathrm{GL}(n,\D)
$.
\begin{lemma}[{cf.~	 \cite[Theorem 15.1.1,  Theorem 5.5.3]{rodman}}] \label{lem-Jordan-M(n,D)}
For every unipotent element $A \in  \mathrm{GL}(n,\D)$,  there is an invertible matrix $S \in  \mathrm{GL}(n,\D)$ such that $SAS^{-1}$ has the following  form: 
\begin{equation}\label{eq-Jordan-unipotent}	SAS^{-1} =\mathrm{I}_{m_0} \, \oplus \,  \mathrm{J}(1,  \,  m_1) \, \oplus  \, \cdots \,  \oplus  \, \mathrm{J}(1, \, m_k),
\end{equation}
where $ m_i \in \N,$  for all $ i \in \{0, 1,2, \dots, k\}.$
The form  (\ref{eq-Jordan-unipotent}) is uniquely determined by $A$ up to a permutation of diagonal blocks.
\end{lemma}

Now we recall a well-known result, which gives equivalence between reversible and strongly reversible elements in ${\rm GL }(n,\D) $ for $\mathbb{D} = \R$ or $ \C$.
\begin{proposition} [cf.~{\cite[ Theorems 4.7]{OS}}]  \label{Prop-str-rev-GL-R-C}
Let $A \in {\rm GL }(n,\D) $,  where $\mathbb{D} = \R$ or $ \C$.  Then $A$ is reversible  in   $ {\rm GL }(n,\D) $ if and only if $A$ is strongly reversible in   $ {\rm GL }(n,\D) $.
\end{proposition}
We would like to mention that the above equivalence does not hold  for  the case $\D =\H$, e.g., $A = (\ib) \in {\rm GL }(1,\H)$ is reversible but not strongly reversible  in ${\rm GL }(1,\H)$.

\section{Reversibility in the affine group $\mathrm{Aff}(n,\D)$ } \label{sec-affine-6}
Consider the affine space $\D^n$,  where $\D =\R,\C $ or $\H$.  Let $ \mathrm{Aff}(n,\D)$  denote the affine group of all invertible affine transformations from $\D^n$ to $\D^n$.  Each element  $g = (A,v)$ of $\mathrm{GL}(n,\D)\ltimes \D^n $ acts on  $\D^n$ as affine transformation  
$$g(x) = A(x) + v\,,$$ where $A \in {\GL}(n,\D) $ is called the \textit{linear part }of $g$ and $v \in \D^n$ is called the \textit{translation part} of $g$.    This action identifies the affine group $\mathrm{Aff}(n,\D)$ with $ \mathrm{GL}(n,\D)\ltimes \D^n $.  We can embed $\D^n$  into $\D^{n+1}$  as the plane
$ \mathbf{P} := \{(x,1) \in \D^{n+1} \mid x \in \D^n \}$. 
Consider the embedding $\Theta : \mathrm{Aff}(n,\D) \longrightarrow \mathrm{GL}(n+1,\D)$ defined as

\begin{equation}  \label{eq-embedding-affine}
	\Theta ( (A,v)) = \begin{pmatrix} A   &  v \\
		\mathbf{0} & 1  \\ 
	\end{pmatrix}\,,
\end{equation} 
where  $\mathbf{0}$  is the zero vector in $\D^n$.  Note that action of  $ \Theta  ( \mathrm{Aff}(n,\D))$ on the plane $\mathbf{P}$ is exactly the same as the action of $\mathrm{Aff}(n,\D)$ on $\D^n $.   In this section,  we will classify reversible and strongly reversible elements in the affine group $\mathrm{Aff}(n,\D)$.  
We begin with an example. 
\begin{example}\label{examp-str-rev-translation} Let $g= (\mathrm{I}_n, v) \in \mathrm{Aff}(n,\D) $.  Consider $g_1= (-\mathrm{I}_n, \mathbf{0})  $ and $g_2=  (-\mathrm{I}_n, -v)$ in $\mathrm{Aff}(n,\D)$.  Then $g_1$ and $g_2$ are  involutions in $\mathrm{Aff}(n,\D) $ such that 
$$ g= g_1 \, g_2, \hbox{ i.e., }  (\mathrm{I}_n, v)  = (-\mathrm{I}_n, \mathbf{0}) \, (-\mathrm{I}_n, -v).$$
Hence,  $g$ is strongly reversible in $\mathrm{Aff}(n,\D) $.
\qed \end{example}

In the next result, we obtain necessary and sufficient conditions for the reversible elements  in  $  \mathrm{Aff}(n,\D) $.  

\begin{lemma}\label{lem-affine-cond-rev}
Let $g = (A,v) \in  \mathrm{Aff}(n,\D) $ be an arbitrary element.  Then $g$ is reversible in $\mathrm{Aff}(n,\D) $ if and only if there exists an element $h = (B,w) \in  \mathrm{Aff}(n,\D) $ such that both the following conditions hold:
\begin{enumerate}
\item \label{cond-rev-aff-1} $BAB^{-1} = A^{-1}$,
\item \label{cond-rev-aff-2} $(A^{-1} - \mathrm{I}_n)(w) = (A^{-1} + B)(v)$.
\end{enumerate}
\end{lemma}
\begin{proof}
Note that $g^{-1}(x) = A^{-1}(x) - A^{-1}(v),  \hbox {and }  h^{-1}(x) = B^{-1}(x) - B^{-1}(w)$ for all $x \in \D^n$.  This implies for all  $x \in \D^n$,  we have
\begin{equation*}
hgh^{-1}(x) = h (AB^{-1}(x) - AB^{-1}(w)+v) = BAB^{-1}(x) - BAB^{-1} (w) +B(v) + w.
\end{equation*}
Therefore,  $hgh^{-1} = g^{-1} \Leftrightarrow BAB^{-1} = A^{-1},   \hbox {and }  - A^{-1}(v) = - BAB^{-1} (w) +B(v) + w$.
This proves the lemma.
\end{proof}

The following lemma gives necessary and sufficient conditions for the strongly  reversible elements in $  \mathrm{Aff}(n,\D) $.  
\begin{lemma}\label{lem-affine-cond-str-rev}
Let $g = (A,v) \in  \mathrm{Aff}(n,\D) $ be an arbitrary element.  Then $g$ is  strongly reversible in $\mathrm{Aff}(n,\D) $ if and only if there exists an element $h = (B,w) \in  \mathrm{Aff}(n,\D) $ such that both the following conditions hold:
\begin{enumerate}
\item \label{cond-str-rev-aff-1} $BAB^{-1} = A^{-1},$  and $ B^2 = \mathrm{I}_n$,
\item \label{cond-str-rev-aff-2}$(B+\mathrm{I}_n)(w)=\mathbf{0}$,  and $(B+A^{-1})(w-v)=\mathbf{0}$.
\end{enumerate}
\end{lemma}
\begin{proof}
Note that  $h = (B,w) \in  \mathrm{Aff}(n,\D) $ is an involution  if and only if  $h^2(x) = B^2(x)+ B(w) + w = x \hbox { for all } x \in \D^n$.  This implies that $ B^2 = \mathrm{I}_n$,  and   $(B + \mathrm{I}_n)(w)=0$. Further,  in view of \lemref{lem-affine-cond-rev},  $hgh^{-1}= g^{-1} $ if and only if  conditions (\ref{cond-rev-aff-1}) and (\ref{cond-rev-aff-2}) of \lemref{lem-affine-cond-rev}  hold.  Observe that equation  $(B + \mathrm{I}_n)(w)=0$ and equation $(A^{-1} - \mathrm{I}_n)(w) = (A^{-1} + B)(v)$ implies $(B+A^{-1})(w-v)=\mathbf{0}$. 
This proves the lemma.
\end{proof}

\subsection{Conjugacy in  the affine group $\mathrm{Aff}(n,\D)$}

In the affine group $\mathrm{Aff}(n,\D)$, up to conjugacy,  we can consider every element in a more simpler form, which is  demonstrated in the next lemma. Recall that a unipotent element  $U \in \mathrm{GL}(n,\D)$  has only $1$ as an eigenvalue.

\begin{lemma}\label{lem-affine-conj-form}
Every element $g$ in $ \mathrm{Aff}(n,\D)$, up to conjugacy, can be written as $g = (A, v)$ such that $A = T \oplus U$,  where $T \in \mathrm{GL}(n-m,\D)$,  $U \in \mathrm{GL}(m,\D)$ such that $T$ does not have eigenvalue $1$, $U$ has only $1$ as eigenvalue,  and $v$ is of the form $v = [0,0,\dots,0,v_1,v_2,\dots,v_m] \in \D^n$, where  $0\leq m \leq n$ is the multiplicity of eigenvalue $1$ of the linear part of $g$.  Further, if $1$ is not an eigenvalue of the linear part of $g$ (i.e., $m=0$),  then up to conjugacy, $g$ is of the form $g = (A, \mathbf{0} )$.
\end{lemma}

\begin{proof}
Let $g \in  \mathrm{Aff}(n,\D) $ be an arbitrary element.  In view of the Jordan decomposition in  $\mathrm{GL}(n,\D)$, after conjugating $g$ by a suitable element $(B,\mathbf{0}) \in  \mathrm{Aff}(n,\D)$, we can assume $g = (A,w)$ such that $A = T  \oplus  U$,  where $T \in \mathrm{GL}(n-m,\D)$ does not have eigenvalue $1$ and $U \in \mathrm{GL}(m,\D)$ is unipotent.  There are two possible cases:
\begin{enumerate}
\item Suppose $1$ is not an eigenvalue of $A$.So the linear transformation $A -  \mathrm{I}_n$ is invertible. Therefore, we can choose $x_o =(A -  \mathrm{I}_n )^{-1} (w)  \in \D^n$. Consider $h = (\mathrm{I}_n, x_o) \in \mathrm{Aff}(n,\D)$.  For all $x \in \D^n$, we have
$$ hgh^{-1} (x) = hg(x- x_o) = h({A}x - {A}x_o + w ) ={A}x + w - ({A}-{\rm I}_n)x_o.$$
This implies $hgh^{-1} (x)= {A}(x) + \mathbf{0} $ for all $x \in \D^n$, since $x_o =({A} -  {{\rm I}_n} )^{-1} (w)$. 
\item Let $1$ be an eigenvalue  of ${A}$.  In this case $m >0$ and 
${A} - {{\rm I}_n}$ has rank $ n-m  < n$. So we can choose an element $u \in \D^n$ having the last $m$ coordinates zero such that $ [({A} - {{\rm I}_n}  )(u)]_i = w_i \hbox{ for all }  \ 1  \leq i \leq n-m, \hbox{where } w = [w_i]_{1  \le i \le n}$. Let $v = w{ - }({A} {-}  {{\rm I}_n} )(u)$. 
Then $v = [0,0, \dots, 0,w_{n-m+1}, w_{n-m+2},\dots,w_n ] \in \D^n$.
Now consider $h = (  \mathrm{I}_n,u) \in \mathrm{Aff}(n,\D)$.  For all $x \in \D^n$,  we have
$$hgh^{-1}(x) = hg(x-u) = h(Ax - Au +w) = Ax + w-(A - \mathrm{I}_n)(u) = Ax +v.$$ 
\end{enumerate}
This completes the proof.
\end{proof}

\begin{remark}
The idea of the above proof is in the same line of arguments as in \cite[Lemma 3.1]{GL}. But here, we have to deal with the subtle situation when the linear part of affine transformations contains a unipotent Jordan block.
\qed \end{remark}

\subsection{Elements in $\mathrm{Aff}(n,\D)$ having  a fixed point}
Recall that if the linear part of an element in $\mathrm{Aff}(n,\D) $ does not have eigenvalue $1$, then it will have a fixed point in $\D^n$. In this case, the classification of reversible  and  strongly reversible elements in $\mathrm{Aff}(n,\D) $ follows from the corresponding classification in $\mathrm{GL}(n,\D)$. 

\begin{proposition}\label{prop-affine-rev-str-rev-type-1} 
Let $g  = (A,v) \in  \mathrm{Aff}(n,\D) $ be an arbitrary element such that $1$ is not an eigenvalue of the linear part $A$ of $g$. Then $g$ is reversible (resp.  strongly reversible) in $\mathrm{Aff}(n,\D) $ if and only if $A$ is reversible (resp.  strongly reversible) in $\mathrm{GL}(n,\D)$.  Further,  for $\D = \R $ or $\C$,  the following are equivalent.
\begin{enumerate}
\item $g$ is  reversible in $\mathrm{Aff}(n,\D) $.
\item $g$ is strongly reversible in $\mathrm{Aff}(n,\D) $.	\end{enumerate}
\end{proposition}

\begin{proof}
Using \lemref{lem-affine-conj-form}, up to conjugacy, we can assume $g = (A, \mathbf{0} )$.  The proof now follows from  \propref{Prop-str-rev-GL-R-C}.
\end{proof}

\subsection{Elements in $\mathrm{Aff}(n,\D)$ with unipotent linear part}\label{unisub}
In this section, we shall use the adjoint reality approach introduced in \cite{GM} to show that every element of $\mathrm{Aff}(n,\D)$  with a unipotent linear part is strongly reversible.  
In view of  \lemref{lem-affine-conj-form} and \propref{prop-affine-rev-str-rev-type-1},  classification of reversible and strongly reversible elements in $\mathrm{Aff}(n,\D) $ reduces to the case when the linear part of the affine group element is unipotent.

In view of \lemref{lem-Jordan-M(n,D)}, every unipotent element in $\mathrm{GL}(n,\D)$ can be written as direct sum of unipotent Jordan blocks; see Equation \eqref{eq-Jordan-unipotent}. Therefore,  it is enough to consider the case when the linear part of an element $g  \in  \mathrm{Aff}(n,\D) $ is equal to the unipotent Jordan block $\mathrm{J}(1,n)$. We will show that  $ g =(\mathrm{J}(1,  \,  n),v) \in  \mathrm{Aff}(n,\D) $ is strongly reversible in $\mathrm{Aff}(n,\D) $ for all $v \in \D^n$ and $n\in \N$. In the following example, we will illustrate this for the case $n=6$ by constructing an explicit involution which conjugate $g$ to $g^{-1}$.

\begin{example}\label{exam-affine-str-rev-unipotent}
Let  $g  = (A,v) \in  \mathrm{Aff}(6,\D) $ be such that $A =  \mathrm{J}(1,  \,  6) \in \mathrm{GL}(6,\D) $,  where  $ \D = \R, \C$ or $\H$.  We will show that  $g$ is strongly reversible in $\mathrm{Aff}(6,\D) $.
	
$	\hbox{Here}, 	A^{-1}  = \begin{pmatrix}
		1 & -1  &1 &-1&1 &-1  \\
		& 1 &  -1&1&-1&1   \\
		&   &  1&-1&1&-1  \\
		&   &  &1&-1&1  \\
		&   &    &  & 1 &-1   \\
		&   &    &  & &1   
	\end{pmatrix} .$
Let $B := \begin{pmatrix}
		1 & 4  &6 &4&1 &0  \\
		& -1 &  -3&-3&-1&0   \\
		&   &  1&2&1&0  \\
		&   &  &-1&-1&0  \\
		&   &    &  & 1 &0   \\
		&   &    &  & &-1   
	\end{pmatrix}$ be an element of  ${\rm GL }(6,\D)$. Note that $B$  is  an involution in ${\rm GL }(6,\D)$ and it conjugates $A$ to $A^{-1}$. Further, we have

\begin{equation}
		B+  \mathrm{I}_6 = \begin{pmatrix}
			2 & 4  &6 &4&1 &0  \\
			& 0 &  -3&-3&-1&0   \\
			&   &  2&2&1&0  \\
			&   &  &0&-1&0  \\
			&   &    &  & 2 &0   \\
			&   &    &  & &0
		\end{pmatrix},   \quad  B  +A^{-1} = \begin{pmatrix}
			2 & 3  &7 &3&2 &-1  \\
			& 0 &  -4&-2&-2&1  \\
			&   &  2&1&2&-1  \\
			&   &  &0&-2&1  \\
			&   &    &  &2 &-1   \\
			&   &    &  & &0  
		\end{pmatrix}. 
	\end{equation}
Note that both the matrices $B + \mathrm{I}_6$ and $B + A^{-1}$ have the same rank, which is equal to 3. Moreover, their corresponding diagonal entries are equal. Now, consider $h  = (B,w) \in  \mathrm{Aff}(6,\D) $, where $w \in \D^n$ is defined as:
	\begin{equation}
		w = \begin{pmatrix}
			4v_1+6v_2+10v_3+4v_2   \\
			-2v_1-3v_2-7v_3-3v_4 \\
			2v_3+v_4   \\
			-2v_3-v_4   \\
			0 \\
			v_6-2v_5
		\end{pmatrix}.
	\end{equation}
Then $h$ satisfies all the conditions of  \lemref{lem-affine-cond-str-rev}.  Therefore,  $h$ is an involution such that $hgh^{-1} =g^{-1}$.  Hence,  $g$ is strongly reversible in $\mathrm{Aff}(6,\D) $.
	\qed \end{example}

The complexity of computation involved in Example \ref{exam-affine-str-rev-unipotent} increases as $n$  (size of the Jordan block) increases if we follow the above approach. Therefore, when the linear part of $g  \in  \mathrm{Aff}(n,\D) $ is $\mathrm{J}(1, n)$,  generalizing the above construction to find reversing involution for $g $ seems to be difficult.  We will choose a different path to avoid the computational difficulties and give a significantly simpler proof by considering adjoint reality in the Lie algebra set-up; see \lemref{lem-affine-str-rev-unipotent-block}.

First, let's introduce some notation that will be used in the next part of this section. As before, let $ \D^n $ be the right $\D $-vector space. Consider $ \D^n $ as an abelian Lie algebra. Then $ {\rm Der}_\D\D^n \, \simeq \, \g\l(n,\D)$. Thus we can make the semi-direct product on  $ \g\l(n,\D) \oplus_\iota \D^n$ by setting $ [(A,0), (0, v) ]:= (0, Av) $; see \cite[Chapter 1, \S 4, Example 2]{Kn} for more details. As done for $ {\rm Aff} (n,\D) $ in Equation \eqref{eq-embedding-affine}, consider the embedding  
\begin{equation*} \Psi  \colon \g\l(n,\D) \oplus_\iota \D^n \, \lto \, \g\l(n+1, \D) \,\    \text{ given by }\ 
	\  \Psi ( (X,w)) = \begin{pmatrix} X   &  w \\
		\mathbf{0} & 0  
	\end{pmatrix}\,.
\end{equation*} 
Then the image has the usual Lie algebra structure, and  $\mathfrak{aff}(n, \D) := \g\l(n,\D) \oplus_\iota\D^n$ is the Lie algebra of the linear Lie group $\mathrm{Aff}(n,\D) $.  
Note that the adjoint action of  $ G:=\mathrm{Aff}(n,\D) $ on its Lie algebra $\g:=  \mathfrak{aff}(n, \D)$ is given by 
\begin{align}\label{Ad-action-affine}
{\rm Ad}\colon G \times \g \lto \g \quad ; \quad {\rm Ad}(A,v)\!\cdot\!(X,w) \,=\, \big(AXA^{-1},\, -(AXA^{-1})v + Aw \big).
\end{align}

Now we recall the  notion of adjoint reality for a linear Lie group $G$, which was introduced in \cite{GM}.  
The adjoint action of a linear Lie group $G$ on its Lie algebra $\g$  is given by the conjugation, 
i.e., ${\rm Ad}(g)X:=gXg^{-1}$.    An element $X\in \g$ is called {\it $ {\rm Ad}_G$-real}  if $-X =gXg^{-1} $ for some $g\in G$.  An  $ {\rm Ad}_G$-real element $X\in \g$ is called {\it strongly $ {\rm Ad}_G$-real } if $-X = \tau X \tau^{-1} $ for some involution   $\tau\in G$; see \cite[Definition 1.1]{GM}. Observe that if $-X =gXg^{-1} $ for some $g\in G$,  then  $(\exp (X))^{-1} =g\exp (X)g^{-1} $. Thus, if $X\in \g$ is  $ {\rm Ad}_G$-real (resp. strongly $ {\rm Ad}_G$-real), then $\exp (X)$ is reversible (resp. strongly reversible) in $G$, \cite[Lemma 2.1]{GM}. But the converse is not  true in general. For example,  $X= \mathrm{diag}( 2 \pi \ib, \pi \ib ) \in   \mathfrak{gl}(2, \C)$ is not $ {\rm Ad}_{\mathrm{GL}(2, \C)}$-real,  but $g =\mathrm{diag}( 1, -1 ) = \exp (X) \in   \mathrm{GL}(2, \C)$ is  reversible.

We will investigate the   $ {\rm Ad}_{\mathrm{Aff}(n,\D)}$-real elements in the Lie algebra $\mathfrak{aff}(n, \D)$. Next result gives necessary and sufficient conditions for the strongly $ {\rm Ad}_{\mathrm{Aff}(n,\D)}$-real elements in $ \mathfrak{aff}(n, \D)$. This can be thought of as a Lie algebra version of  Lemma \ref{lem-affine-cond-str-rev}.

\begin{lemma}\label{lem-affine-cond-str-rev-Lie-alg}
Let $ (N,x) \in   \mathfrak{aff}(n, \D)$ be an arbitrary element.  Then $(N,x)$ is  strongly $ {\rm Ad}_{\mathrm{Aff}(n,\D)  }$-real
if and only if there exists an element $h = (B,w) \in  \mathrm{Aff}(n,\D) $ such that both the following conditions hold:
	\begin{enumerate}
		\item \label{cond-str-rev-aff-1-lie-alg}
		$BNB^{-1} = -N,$  and $ B^2 = \mathrm{I}_n$,
		\item \label{cond-str-rev-aff-2-lie-alg}
		$(B+\mathrm{I}_n)(w)=\mathbf{0}$,  and $N(w)\,=\,-(B+\mathrm{I}_n)(x)$.
	\end{enumerate}
\end{lemma}
\begin{proof}
We omit the proof as it is identical to that of \lemref{lem-affine-cond-str-rev}.
\end{proof}

The following result will be used in proving  \lemref{lem-affine-str-rev-unipotent-block}.
\begin{lemma}\label{lem-affine-str-rev-nilpotent-block}
Let  $ (N,x) \in  \mathfrak{aff}(n, \D)$ such that $N= \mathrm{J}(0,  \,  n)$,  where  $ \D = \R, \C$ or $\H$.  Then  $ (N, x)$ is strongly $ {\rm Ad}_{\mathrm{Aff}(n,\D)  }$-real.
\end{lemma}

\begin{proof}
For the element  $ (N, x)$, consider $ B:= {\rm diag}((-1)^{n}, (-1)^{n-1}, \dots, 1,-1)_{n \times n} $.  Then condition (\ref{cond-str-rev-aff-1-lie-alg}) of \lemref{lem-affine-cond-str-rev-Lie-alg}  holds.  Further, by choosing the diagonal matrix $B$, the last row of $N$ and $B+\mathrm{I}_n$ are equal to zero vector in $\D^n$.  This implies that for every $x \in \D^n$, the last coordinate of $B+\mathrm{I}_n (x)$ is zero. Since the rank of $N$ is $n-1$, so equation $Nw\,=\,-(B+\mathrm{I}_n)(x)$ is consistent for given  $x \in \D^n$ and has a solution.  
To prove this lemma,  it is sufficient to choose $ w\in \D^n $ so that the condition (\ref{cond-str-rev-aff-2-lie-alg})  of Lemma \ref{lem-affine-cond-str-rev-Lie-alg} holds. This can be done in the following way:
\begin{enumerate}
		\item Let $n$ be even.   Then for $x = [x_{k}]_{n \times 1} \in \D^n$,   take $w = [w_{k}]_{n \times 1} \in \D^n $ such that
		\begin{equation*}
			w_{2k-1}= 0,  \hbox{ and } w_{2k} =-2 \, x_{2k-1},  \hbox{ where } k \in \{1,2,\dots,\frac{n}{2} \}.
		\end{equation*} 
		Here, we get unique $w$  depending on $v$ for our choice of $B$.
		\item Let $n$ be odd.   Then for $x = [x_{k}]_{n \times 1} \in \D^n$,   take $w = [w_{k}]_{n \times 1} \in \D^n $ such that
		\begin{equation*}
			w_1 \in \D,  \,  w_{2k}= 0,  \hbox{ and } w_{2k+1} =-2 \, x_{2k},  \hbox{ where } k \in \{1,2,\dots,\frac{n-1}{2} \}.
		\end{equation*} 
		Here,  for our choice of $B$, we get no condition on $w_1$.  
	\end{enumerate}
Then in view of Lemma \ref{lem-affine-cond-str-rev-Lie-alg},  the element $(N, x)$ is strongly $ {\rm Ad}_{\mathrm{Aff}(n,\D)  }$-real.  Hence, the proof follows.
\end{proof}

The following lemma demonstrates that affine transformations with linear part conjugate to a unipotent Jordan block are strongly reversible.

\begin{lemma}\label{lem-affine-str-rev-unipotent-block}
Let  $ (A,v) \in  \mathrm{Aff}(n,\D) $ such that $A= \mathrm{J}(1,  \,  n)$,  where  $ \D = \R, \C$ or $\H$.  Then $g$ is strongly reversible in $\mathrm{Aff}(n,\D) $.
\end{lemma}

\begin{proof}
Let $ N:=  \mathrm{J}(0,  \,  n) \in \g\l(n,\D)$.
Then $( \sigma, y) \exp((N,x)) (\sigma,y)^{-1}  \,=\,(A, v) $ for some $(\sigma, y) \in  \mathrm{Aff}(n,\D)$. Recall that the Lie algebra $\mathfrak{aff}(n, \D) = \g\l(n,\D) \oplus_\iota\D^n $.
Using Lemma \ref{lem-affine-str-rev-nilpotent-block},  we have  that $ (N, x)  \in  \mathfrak{aff}(n, \D)$ is strongly $ {\rm Ad}_{\mathrm{Aff}(n,\D)  }$-real. 
Let $ (\alpha, z) \in  \mathrm{Aff}(n,\D)$ be an involution so that  $ (\alpha, z) (N,x) (\alpha, z) \,=\,-(N, x)$. By taking the exponential, we have that $ (\alpha, z) \exp((N,x)) (\alpha, z)^{-1} \,=\,\exp(-(N, x))$. Let $g:= ( \sigma, y) (\alpha, z)( \sigma, y)^{-1} $. Then $ g $ is an involution in ${\rm Aff }(n,\D) $ and  $ g(A,v) g^{-1}\,=\, (A,v)^{-1}$; see \cite[Lemma 2.1]{GM}. This completes the proof.
\end{proof}

The next result follows from \lemref{lem-affine-str-rev-unipotent-block},  which will be crucially used in the proof of   \thmref{thm-affine-main-theorem}.
\begin{proposition}\label{prop-affine-str-rev-unipotent}
Let  $g  = (A,v) \in  \mathrm{Aff}(n,\D) $ such that $A$ is a unipotent matrix,  where  $ \D = \R, \C$ or $\H$.  Then $g$ is strongly  reversible in $\mathrm{Aff}(n,\D) $ and consequently $g$ is also reversible in $\mathrm{Aff}(n,\D) $.
\end{proposition}
\begin{proof}
In view of \lemref{lem-Jordan-M(n,D)}, up to conjugacy in $\mathrm{GL}(n,\D)$, we can  assume  $A$ as in Jordan form given by Equation \eqref{eq-Jordan-unipotent}.
Using   \lemref{lem-affine-str-rev-unipotent-block} and Example \ref{examp-str-rev-translation},  we can construct  a suitable $h  = (B,w) \in  \mathrm{Aff}(n,\D) $ such that $hgh^{-1}=g^{-1}$.  Hence, $g$ is strongly  reversible in $\mathrm{Aff}(n,\D) $.  This completes the proof.
\end{proof}

\subsection{Proof of \thmref{thm-affine-main-theorem}} 
Let $g \in {\rm Aff}(n,\D)$  be an arbitrary element. Using  \lemref{lem-affine-cond-rev} and  \lemref{lem-affine-cond-str-rev},  it follows that if $g$ is reversible (resp.  strongly reversible) in $\mathrm{Aff}(n,\D) $ then $A$ is reversible (resp.  strongly reversible) in $\mathrm{GL}(n,\D)$.  

Conversely,  using   \lemref{lem-affine-conj-form},  up to conjugacy,  we can assume that $ g = (A,v)  \in {\rm Aff}(n,\D)$ such that
\begin{equation}\label{eq-affine-main-theorem-1}
	A =\begin{pmatrix} T   &   \\
		& U
	\end{pmatrix},  \quad v = \begin{pmatrix}  \mathbf{0}_{n-m}   \\
		\tilde{v} 
	\end{pmatrix}, 
\end{equation}
where $0 \leq m \leq n$, $ \mathbf{0}_{n-m} $ denotes the zero vector in $\D^{n-m}$ and $T \in \mathrm{GL}(n-m,\D)$,  $U \in \mathrm{GL}(m,\D)$ such that $T$ does not have eigenvalue $1$, $U$ has only $1$ as eigenvalue and $  \tilde{v}  = [v_1,v_2,\dots,v_m] \in \D^m$.  
Here,  $T$ and $U$ do not have a common eigenvalue. This implies that
if $B \in \mathrm{GL}(n,\D)$ is such that $BAB^{-1}=A^{-1}$,  then $B$ has the following form
$$B=\begin{pmatrix} B_1   &   \\
	& B_2
\end{pmatrix},  \, \hbox{where } B_1 \in \mathrm{GL}(n-m,\D),  B_2 \in \mathrm{GL}(m,\D).
$$
Therefore,  if $A$ is reversible (resp.  strongly reversible) in $\mathrm{GL}(n,\D)$,  then $T \in \mathrm{GL}(n-m,\D)$ and $U \in \mathrm{GL}(m,\D)$  are reversible (resp.  strongly reversible).  Consider $h= (U,\tilde{v}) \in {\rm Aff}(m,\D)$, where $U$ is  a unipotent matrix.  Then  \propref{prop-affine-str-rev-unipotent} implies that $h$ is strongly reversible in  ${\rm Aff}(m,\D)$.  Proof of the converse part now follows from Equation  \eqref{eq-affine-main-theorem-1}.

Further,  for the case $\D= \R $ or $\C$,   \propref{Prop-str-rev-GL-R-C} implies that $g$ is reversible  in   $ {\rm Aff }(n,\D) $ if and only if $g$ is strongly reversible in  $ {\rm Aff }(n,\D) $. This completes the proof.
\qed

\section{Product of involutions in $\mathrm{Aff}(n,\D) $ } \label{sec-involution-length}

In this section, we  investigate the involution length in the group  $\mathrm{Aff}(n,\D) $. We shall begin by recalling the basic concept of determinant for matrices over $ \H $. 
For  $A \in  \mathrm{M}(n,\H)$, let  $ A = (A_1) + (A_2) \jb $ for some  $ A_1,  A_2 \in \mathrm{M}(n,\C)$.  Consider the embedding $ \Phi:  \mathrm{M}(n,\H)  \longrightarrow  \mathrm{M}(2n,\C)$ defined as 
\begin{equation}\label{eq-embedding-phi}
	\Phi(A) = \begin{pmatrix} A_1   &  A_2 \\
		- \overline{A_2} & \overline{A_1}  
	\end{pmatrix}\,, 
\end{equation}
where $ \overline{A_j} $ denotes the complex conjugate of $ A_j $.

\begin{definition}\label{def-det}
For $A \in  \mathrm{M}(n,\H)$,  determinant  of $A$ is defined as the determinant  of corresponding matrix $ \Phi(A)$, i.e., $ {\rm det}(A):=  {\rm det}(\Phi(A))$,  where $\Phi$ is as defined in Equation   \eqref{eq-embedding-phi};  see \cite[Section 5.9]{rodman}. In view of the \textit{Skolem-Noether theorem}, the above definition is independent of the choice of the chosen embedding $ \Phi$. \qed
\end{definition}

Recall that if $h =(B,v) \in \mathrm{Aff}(n,\D)$  is an involution,  then $B$ has to be an involution in $\mathrm{GL}(n,\D)$; see \lemref{lem-affine-cond-str-rev}. 
If an element of $\mathrm{GL}(n,\D)$ is a product of involutions, then necessarily its determinant is either $1$ or $ -1$.   
Product of involutions in $\mathrm{GL}(n,\D)$ has been studied in \cite{GHR} and  \cite[Section 4.2.4]{OS} for the case  $ \D=\R$ or $\C$.

In the next result, we investigate the product of involutions in $\mathrm{GL}(n,\D)$. 
\begin{lemma}\label{lem-involution-length-GL(n,D)}
Let  $ \D=\R, \C $ or $\H$.  Every element of $\mathrm{GL}(n,\D)$ with determinant $1$ or $-1$ can be written as a product of at most four involutions.
\end{lemma}
\begin{proof}
Using the Jordan decomposition over $\H$,  up to conjugacy, we can assume that every element of $\mathrm{GL}(n,\H)$  is in $\mathrm{GL}(n,\C)$; see {\cite[  Theorem 5.5.3]{rodman}}.    The proof now follows from \cite[Theorem 4.9]{OS}.
\end{proof}

\begin{remark}\label{rem-equvi-str}
Note that an element of a group $G$ is strongly reversible if and only if it can be expressed as a product of two involutions in $G$; see \cite[Proposition 2.12]{OS}.  
\qed
\end{remark}

Next, we will prove \thmref{thm-affine-involution-lenght}.

\textbf{Proof of \thmref{thm-affine-involution-lenght}.}
~Let  $ g = (A,v) \in {\rm Aff}(n,\D) $ be such that $ \mathrm{det}(A) \in \{-1, 1\}$. Then using   \lemref{lem-affine-conj-form},  up to conjugacy, we can assume that
\begin{equation}\label{eq-affine-involution-lenght-1}
	A =\begin{pmatrix} T   &   \\
		& U
	\end{pmatrix},  \quad v = \begin{pmatrix} \mathbf{0}_{n-m}   \\
		\tilde{v} 
	\end{pmatrix}, 
\end{equation}
where  $T \in \mathrm{GL}(n-m,\D)$ and  $U \in \mathrm{GL}(m,\D)$ such that $T$ does not have eigenvalue $1$ and  $U$ has only $1$ as eigenvalue. Here, $0\leq m \leq n$, $ \mathbf{0}_{n-m} $ denotes the zero vector in $\D^{n-m}$  and $  \tilde{v}  = [v_1,v_2,\dots,v_m] \in \D^m$.  Consider $h = (U,\tilde{v} ) \in {\rm Aff}(m,\D) $.  
Using \propref{prop-affine-str-rev-unipotent},  $h$ is strongly reversible in ${\rm Aff}(m,\D) $.  
Therefore,  in view of  \remref{rem-equvi-str},  there exist involutions $h_1 = (P,u)$ and $h_2 = (Q,w)$  in ${\rm GL}(m,\D) \ltimes \D^m $ such that
\begin{equation}\label{eq-affine-involution-lenght-2}
	h = h_1 \,  h_2.
\end{equation}
Further,  note that $\mathrm{det}(A)= \mathrm{det}(T) \,  \mathrm{det}(U) = \mathrm{det}(T)$.  Thus $T \in {\rm GL}(n-m,\D)$  has  determinant either $1$ or $-1$.  In view of  \lemref{lem-involution-length-GL(n,D)},  we have
\begin{equation}\label{eq-affine-involution-lenght-3}
	T= B_1 \, B_2 \, B_3 \, B_4,
\end{equation}
where $ B_i$ is an involution in ${\rm GL}(n-m,\D)$ for all $i \in \{1,2,3,4\}$. Here, $B_i $ may be equal to $\mathrm{I}_{n-m}$ for some $i \in \{1,2,3,4\}$.
Now consider  the following elements in ${\rm Aff}(n,\D) $:
\begin{itemize}
	\item $f_1 := ( B_1 \,  \oplus \,  \mathrm{I}_m, \, \mathbf{0}_{n}),$
	\item $ f_2 := ( B_2 \,  \oplus \,  \mathrm{I}_m,\, \mathbf{0}_{n}), $
	\item $ f_3 := ( B_3 \,  \oplus \, P, \, \mathbf{0}_{n-m} \,  \oplus \,  u), $
	\item $ f_4 := ( B_4 \, \oplus \, Q, \, \mathbf{0}_{n-m} \,  \oplus  \, w)$. 
\end{itemize} 
From the above construction,  it is clear that $f_1, f_2, f_3, $ and $f_4$ are involutions in ${\rm Aff}(n,\D) $.
Using  Equations \eqref{eq-affine-involution-lenght-1}, \eqref{eq-affine-involution-lenght-2} and \eqref{eq-affine-involution-lenght-3}, we have  $ g= f_1\, f_2\, f_3\, f_4$.  This completes the proof.
\qed

\subsection*{Acknowledgement}
\medskip Gongopadhyay is partially supported by the SERB core research grant 
CRG/2022/003680. 
Lohan acknowledges  support from the CSIR SRF grant, File No.\,: \\ 09/947(0113)/2019-EMR-I. 
Maity is supported by an NBHM PDF during  this work.

\end{document}